\documentclass[11pt]{article}

\title
{Long cycles in locally expanding graphs, with applications}
\author{Michael Krivelevich
\thanks{School of Mathematical Sciences, Raymond and Beverly
Sackler Faculty of Exact Sciences, Tel Aviv University, Tel Aviv,
6997801, Israel. Email: krivelev@post.tau.ac.il. Research supported in
part by USA-Israel BSF grant 2014361.}
}

\usepackage{amsmath,amsthm, amssymb,latexsym, amsfonts}

\begin{document}
\bibliographystyle{plain}
\maketitle
\newtheorem{thm}{Theorem}
\newtheorem{propos}[thm]{Proposition}
\newtheorem{defin}{Definition}
\newtheorem{lemma}{Lemma}[section]
\newtheorem{corol}[lemma]{Corollary}
\newtheorem{thmtool}{Theorem}[section]
\newtheorem{corollary}[thmtool]{Corollary}
\newtheorem{lem}[thmtool]{Lemma}
\newtheorem{prop}[thmtool]{Proposition}
\newtheorem{clm}[thmtool]{Claim}
\newtheorem{conjecture}{Conjecture}
\newtheorem{problem}{Problem}
\newcommand{\Proof}{\noindent{\bf Proof.}\ \ }
\newcommand{\Remarks}{\noindent{\bf Remarks:}\ \ }
\newcommand{\Remark}{\noindent{\bf Remark:}\ \ }
\newcommand{\whp}{{\bf whp}\ }
\newcommand{\prob}{probability}
\newcommand{\rn}{random}
\newcommand{\rv}{random variable}
\newcommand{\hpg}{hypergraph}
\newcommand{\hpgs}{hypergraphs}
\newcommand{\subhpg}{subhypergraph}
\newcommand{\subhpgs}{subhypergraphs}
\newcommand{\bH}{{\bf H}}
\newcommand{\cH}{{\cal H}}
\newcommand{\cT}{{\cal T}}
\newcommand{\cF}{{\cal F}}
\newcommand{\cD}{{\cal D}}
\newcommand{\cC}{{\cal C}}

\begin{abstract}
We provide sufficient conditions for the existence of long cycles in locally expanding graphs, and present applications of our conditions and techniques to Ramsey theory, random graphs and positional games.
\end{abstract}

\section{Introduction and main results}
In this paper, we are concerned with the following question, or rather question type: how long a cycle can we find in a graph $G$, assuming its vertex subsets up to a certain size, or in a certain range of sizes, expand? Of course, one needs to formalize the local expansion assumption to lay down concrete results.

As usually, for a graph $G=(V,E)$ and a vertex set $W\subset V$ we denote by $N_G(W)$ the external neighborhood of $W$ in $G$, i.e.,
$$
N_G(W)=\{v\in V\setminus W: v\mbox{ has a neighbor in }W\}\,.
$$
Using this notation, we can make one more step in specifying our meta-question: suppose we are given that in a graph $G=(V,E)$, the set $N_G(W)$ is large for all subsets $W\subset V$ up to a certain size, or in a certain range of sizes. What can be guaranteed then for the length of a longest cycle in $G$? We urge the reader to notice that conditions on external neighborhoods can be imposed for sets up to a certain size, independent of the order of the graph; thus in a sense assumptions of this sort are of {\em local} nature and can be viewed as {\em local expansion} conditions. (See \cite{HLW06, KS06} for a general background on expanders and pseudo-random graphs.)

The above setting appears to be quite natural; moreover, results of this sort can be handy in a variety of scenarios and applications (of which we will see some later in the paper). Yet, to the best of our knowledge there are few results of this sort available in the literature.

The most notable previous relevant result of this type is probably that of Brandt, Broersma, Diestel and Kriesell \cite{BBDK06}. In order to state their result and for future use in our paper, let us introduce the following notation. For a positive integer $k$ and a positive real $\alpha$, a graph $G$ is called {\em $(k,\alpha)$-expander} if $|N_G(W)|\ge \alpha|W|$ for all subsets $W\subset V$ of size $|W|\le k$. Brandt et al. proved in particular that if $G$ is a $(k,\alpha)$-expander and $\alpha\ge 2$, then $G$ contains a cycle of length at least $(\alpha+1)k$ (see Lemma 2.7 and the proof of Theorem 2.4 in their paper). The assumption $\alpha\ge 2$ plays a crucial role in their argument, as it allows to utilize the famed rotation-extension technique of P\'osa \cite{Pos76}; the argument essentially breaks down completely even for $\alpha=1.99$, say, in the assumption above.

One should notice that finding long paths, rather than cycles, in locally expanding graphs is a much easier task, see, e.g., Chapter 1 of \cite{KPPM16} for a survey of some techniques and results available. We will encounter substantial differences between paths and cycles in this context later in the paper. More generally, local expansion conditions have been used in results about embedding large bounded degree trees, see, e.g., \cite{FP87,Hax01,BCPS10}.

Our first result shows the existence of long cycles in locally expanding graphs.

\begin{thm}\label{th1}
Let $k>0,t\ge 2$ be integers. Let $G=(V,E)$ be a graph on more than $k$ vertices, satisfying:
$$
|N_G(W)|\ge t,\quad \mbox{for every } W\subset V,\  \frac{k}{2}\le |W|\le k\,.
$$
Then $G$ contains a cycle of length at least $t+1$.
\end{thm}

Hence for {\em every} positive $\alpha$, a $(k,\alpha)$-expander contains a cycle of length at least $\frac{\alpha k}{2}$. The linear dependence on $\alpha$ is optimal in the range $0<\alpha<1$, as shown by the example of the complete bipartite graph with parts of size $k$ and $\lceil \alpha k\rceil$; a longest cycle here has length $2\lceil \alpha k\rceil$. The above theorem is a (rather far-reaching) generalization of the classical and simple graph theoretic result, postulating that every graph $G$ of minimum degree $t\ge 2$ contains a cycle of length at least $t+1$; we obtain it by choosing $k=1$ in Theorem \ref{th1}.

Our second result also provides sufficient conditions for the existence of long cycles. It is formulated in quite different terms and assumes in particular local sparseness of a given graph. Though the statement looks very different from the previous theorem, the proof is in fact fairly similar, as the reader will get to see.

\begin{thm}\label{th2}
Suppose reals $c_1>c_2>1$ and a positive integer $k$  satisfy: $\left(\frac{k}{2}-1\right)\left(\left(\frac{c_1}{c_2}\right)^{1/2} -1\right)\ge 2$.
Let $G=(V,E)$ be a graph on more than $k$ vertices, having the following properties:
\begin{enumerate}
\item $|E(G)|\ge c_1|V(G)|$;
\item every subset $R$ of $V(G)$ of size $|R|\le k$ spans at most $c_2|R|$ edges in $G$.
\end{enumerate}
Then $G$ has a cycle of length at least $\left(\frac{k}{2}-1\right)\left(\left(\frac{c_1}{c_2}\right)^{1/2} -1\right)+1$.
\end{thm}

In words, the statement says that if the local density (in sets of size at most $k$) is notably smaller than the global density of $G$, then $G$ contains a long cycle; ``long" here means linear in $k$, for absolute constants $c_1,c_2$.

The conditions assumed by Theorem \ref{th2} may appear rather artificial, but they come in fact quite naturally in the setting of random graphs, as given by the following proposition.

\begin{propos}\label{pr1}
Let $c_1>c_2>1$ be reals. Define $\delta=\left(\frac{c_2}{5c_1}\right)^{\frac{c_2}{c_2-1}}$. Let $G$ be a random graph drawn from the probability distribution $G\left(n,\frac{c_1}{n}\right)$. Then \whp every set of $k\le \delta n$ vertices of $G$ spans fewer than $c_2k$ edges.
\end{propos}

We now proceed to describe applications of our results. The first application is in the realm of Ramsey theory. As customary, for graphs $G,H$ and a positive integer $r$ we write $G\rightarrow (H)_r$ if every $r$-coloring of the edges of $G$ produces a monochromatic copy of $H$. Furthermore, the {\em $r$-color size Ramsey number} $\hat{R}(H,r)$ is defined as the minimal possible number of edges in a graph $G$, for which $G\rightarrow (H)_r$. Size Ramsey numbers is a frequently studied subject in Ramsey theory, see \cite{FS02} for a survey. Here we concentrate actually on size Ramsey numbers of paths, rather than cycles (see \cite{HKL95} for a result on size Ramsey numbers of cycles). In the groundbreaking paper \cite{Bec83} Beck, resolving a \$100 question of Erd\H{o}s, proved that $\hat{R}(P_n,2)=O(n)$, where $P_n$ is an $n$-vertex path. (Recently, there has been a race to improve the multiplicative constant in this bound, see \cite{DP15,L16,DP16}; the best result is due to Dudek and Pra\l at, who showed in \cite{DP16} that $\hat{R}(P_n,2)\le 74n$ for large enough $n$.) Beck's argument is density-type and can easily be adapted to show that $\hat{R}(P_n,r)=O_r(n)$. The next question to ask then is what is the hidden dependence on $r$ in this linear in $n$ bound. This question has recently been addressed by Dudek and Pra\l at; they proved \cite{DP16} that $\hat{R}(P_n,r)\ge \frac{(r+3)r}{4}n-O(r^2)$ for all large enough $n$. In Section \ref{sect5} we will provide an alternative proof of the lower bound of the size Ramsey number of paths. As for the upper bound in terms of $r$, in the same paper it was shown that $\hat{R}(P_n,r)\le 33r4^rn$ for all large enough $n$. These two bounds are quite far apart in terms of dependence on $r$, leaving the exponential gap. Here we nearly close the gap by showing that the dependence on $r$ is (nearly) quadratic.

\begin{thm}\label{th3}
The size Ramsey number $\hat{R}(P_n,r)$ satisfies: $\hat{R}(P_n,r)\le r^{2+o_r(1)}n$, for all sufficiently large $n$.
\end{thm}

The next problem to discuss is positioned at the intersection of Ramsey theory and random graphs. Given a random graph $G$, drawn from the probability distribution $G(n,p)$, how large a monochromatic connected component can one typically find in every $r$-coloring of the edges of $G$? This problem has been considered by  Bohman, Frieze, Krivelevich,
Loh and Sudakov \cite{BFKLS11}, and independently by Sp\"ohel, Steger and Thomas \cite{SST10}. It turns out that the answer is closely related to the threshold for the so-called $r$-orientability. For a positive integer $r$, a graph $G$ is called {\em $r$-orientable} if one can direct its edges in such a way that all out-degrees in the resulting digraph are at most $r$. Obviously if $G=(V,E)$ is $r$-orientable then so is every subgraph of it, and therefore every vertex subset $U\subset V$ of $G$ contains at most $r|U|$ edges and spans a subgraph with average degree at most $2r$. Cain, Sanders and Wormald \cite{CSW07}, and Fernholz and Ramachandran \cite{FR07} discovered independently that the threshold for $r$-orientability in $G(n,p)$ coincides with the number of edges needed to make the $(r + 1)$-core have average
degree above 2r. More concretely, they showed that for any integer $r\ge 2$, there is an explicit threshold
$\psi_r$ such that the following holds. For any $\epsilon>0$, if $p > \frac{\psi_r +\epsilon}{n}$, then \whp $G\sim G(n,p)$ contains a subgraph with average degree at least $2r + \delta$, where $\delta=\delta(\epsilon)>0$ (and is thus not $r$-orientable). On the other hand, if $p < \frac{\psi_r -\epsilon}{n}$, then $G$ is $r$-orientable {\bf whp}. \cite{BFKLS11, SST10} used these results to argue that any $\epsilon>0$, if $p > \frac{\psi_r +\epsilon}{n}$, then \whp $G\sim G(n,p)$ is such that every $r$-coloring of $E(G)$ has a monochromatic connected component of size at least $c(\epsilon)n$, whereas if $p < \frac{\psi_r -\epsilon}{n}$, then \whp the edges of $G\sim G(n,p)$ can be $r$-colored with every color class having all components of order $o(n)$. Here we prove that in the supercritical case $p > \frac{\psi_r +\epsilon}{n}$, the random graph $G(n,p)$ is typically such that every $r$-coloring of its edges produces a linearly long monochromatic cycle:

\begin{thm}\label{th4}
Given any fixed $r\ge 2$, let $\psi_r$ be the threshold for $r$-orientability. Then for any $\epsilon>0$ there exists $\delta>0$ such that if $p(n) \ge \frac{\psi_r +\epsilon}{n}$, then a random graph $G\sim G(n,p)$ is \whp such that every $r$-coloring of $E(G)$ produces a  cycle of length at least $\delta n$ in one of the colors.
\end{thm}

The last group of applications of our results is in positional games. (The reader can consult monographs \cite{Beck-book, HKSS-book} for a general background on this fascinating subject.) Here we consider only games played on the edge set of the complete graph $K_n$. In a {\em Maker-Breaker} game two players, called Maker and Breaker, claim alternately free edges of the complete graph, with Maker moving first. Maker claims one edge at a time, while Breaker claims $b\ge 1$ edges (or all remaining fewer than $b$ edges if this is the last round of the game). The integer parameter $b$ is the so-called {\em game bias}. Maker wins the game if the graph of his edges in the end possesses a given graph theoretic property, Breaker wins otherwise, with draw being impossible. The type of Maker-Breaker games we consider here are cycle creation games. Bednarska and Pikhurko proved in \cite{BP05} that Maker can create a cycle in this game if and only if $b<\lceil n/2\rceil-1$. (See their other paper \cite{BP08} for results on odd and even cycle creation games.) Here we prove that shortly before the game bias $b$ crosses the critical point $n/2$, Maker can in fact create a linearly long cycle.

\begin{thm}\label{th5}
For every $0<\epsilon<1$ there exists $n_0$ such that for all $n>n_0$, Maker has a strategy to create a linearly long cycle in the $(1:b)$ Maker-Breaker game on the edges of the complete graph $K_n$ on $n$ vertices, for all $b\le (1-\epsilon)n/2$.
\end{thm}

It is instructive to observe a very substantial difference between long cycles and long paths in this context. Indeed, Bednarska and \L uczak proved in \cite{BL01} that if $b \ge (1+\epsilon)n$ then Breaker can prevent a component of size larger than $1/\epsilon$ in Maker's graph in the $(1:b)$ Maker-Breaker game on $E(K_n)$. From the other side, Krivelevich and Sudakov showed in \cite{KS13} that for $b \le (1-\epsilon)n$ Maker has a strategy to create a linearly long path. It thus follows that for the (long) path game the threshold bias lies around $b=n$ --- quite different from the (long) cycle game, where it is asymptotically equal to $n/2$.

Finally we consider Client-Waiter long cycle games. Client-Waiter games (also called Chooser-Picker games) have become a relatively popular research subject recently, see, e.g., \cite{Bed13,HKT1,HKT2,DK16}. In a $(1:b)$ {\em Client-Waiter game} played on $E(K_n)$,  in each turn Waiter offers Client a subset of at least one and at most $b + 1$ unclaimed edges,  from which Client claims one, and the rest are claimed by Waiter. The game ends when all edges have been claimed. If Client's graph has a given graph property ${\cal P}$ by the end of the game, then he wins the game, otherwise Waiter is the winner; the game never ends in a draw. Hefetz et al. observed in \cite{HKT1} that it follows more or less immediately from the prior result of  Bednarska-Bzd\c{e}ga, Hefetz, \L uczak and the author \cite{BHKL16} that if $b\ge n/2-1$, then Waiter has a strategy to keep Client's graph acyclic throughout the game. From the other side, Dean and the author argued in \cite{DK16} that for $b \le (1-\epsilon)n/2$ Client has a strategy to create a linearly long path in the Client-Waiter game on $E(K_n)$. Here we strengthen the latter result by going from long paths to long cycles.

\begin{thm}\label{th6}
For every $0<\epsilon<1$ there exists $n_0$ such that for all $n>n_0$, Client has a strategy to create a linearly long cycle in the $(1:b)$ Client-Waiter game on the edges of the complete graph $K_n$ on $n$ vertices, for all $b\le (1-\epsilon)n/2$.
\end{thm}
It thus follows that for the (long) cycle Client-Waiter game the threshold bias is asymptotic to $b=n/2$.

The rest of the paper is structured as follows. In the next section we introduce several technical tools, to be utilized by forthcoming proofs. Section \ref{sect3} is devoted to the proofs of Theorems \ref{th1} and \ref{th2}. In Section \ref{sect4} we present proofs of the results about random graphs (Proposition \ref{pr1} and Theorem \ref{th4}). In Section \ref{sect5} we prove bounds on multi-color size Ramsey numbers of paths. In Section \ref{sect6} we give proofs of results about positional games, Theorems \ref{th5} and \ref{th6}.

\medskip

\noindent{\bf Notation.} Our notation is mostly standard. As stated before, for a graph $G=(V,E)$ and a vertex subset $W\subset V$ we denote by $N_G(W)$ the external neighborhood of $W$ in $G$. Given a rooted tree $T$ with root $r$ and vertex set $V$, for $v\in V$ we denote by $s(T,v)$ the size (number of vertices) of the subtree of $T$ rooted at $v$ (where as usually the subtree of $T$ rooted at $v$ includes exactly those vertices whose path to the root contains $v$). We suppress the rounding notation occasionally to simplify the presentation.

\section{Tools}\label{sect2}
In this section, we gather several notions and technical tools, needed for the proofs in subsequent sections.

\noindent{\bf Finding subtrees of prescribed total size under a vertex.}
 The following simple proposition will be used to find a collection of subtrees close to a prescribed total size, rooted at a vertex of a given rooted tree.
\begin{prop}\label{subtrees}
Let $T=(V,E)$ be a rooted tree on more than $k$ vertices. Then there exists a vertex $v\in V(T)$ and a subset $X$ of the children of $v$ in $T$ such that $\frac{k}{2}\le\sum_{u\in X} s(T,u)\le k$.
\end{prop}

\begin{proof} Let $r$ be the root of $T$. Going down the tree from the root, find a vertex $v$ such that $s(T,v)>k$ but $s(T,u)\le k$ for every child $u$ of $v$ in $T$. Let $X_0$ be the set of chidlren of $v$ in $T$. Then $\sum_{u\in X_0}s(T,u)\ge k$, but $s(T,u)\le k$ for every $u\in X_0$. If there exists a vertex $u\in X_0$ for which $s(T,u)\ge k/2$, take $X=\{u\}$. Otherwise, let $X\subseteq X_0$ be a maximal by inclusion subset of $X_0$ for which $\sum_{u\in X} s(T,u)\le k$. Then $\sum_{u\in X} s(T,u)\ge k/2$, as otherwise one can add an arbitrary vertex $u\in X_0\setminus X$ to $X$ and produce a larger subset whose subtree sizes still sum up to at most $k$ -- a contradiction to the choice of $X$.
\end{proof}

\medskip

\noindent{\bf Depth First Search.} The {\em Depth First Search} is a well known graph exploration algorithm, usually applied to discover
connected components of an input graph.  This algorithm has proven itself particularly suitable for
finding long paths in graphs, see \cite{AKS81, AKS00, BBDK06,BKS12,BKS12-2,Pok14,DP15,L16,DP16} for some examples. Here we will actually use it for finding cycles rather than paths. Since this is a fairly standard graph exploration algorithm, discussed in details in every basic algorithmic book, we will be somewhat succinct in its description below.

Recall that the DFS (Depth First Search) is a graph search algorithm that visits all vertices
of a graph. The algorithm receives as an input a graph $G=(V,E)$; it is
also assumed that an order $\sigma$ on the vertices of $G$ is given, and the algorithm prioritizes vertices
according to $\sigma$. The algorithm maintains three sets of vertices, letting $S$ be the set of vertices whose
exploration is complete, $T$ be the set of unvisited vertices, and $U = V\setminus (S\cup T)$, where the vertices of
$U$ are kept in a stack (the last in, first out data structure). It initiliazes with $S = U = \emptyset$ and $T = V$,
and runs till $U\cup T=\emptyset$. At each round of the algorithm, if the set $U$ is non-empty, the algorithm
queries $T$ for neighbors of the last vertex $v$ that has been added to $U$, scanning $T$ according to $\sigma$. If
$v$ has a neighbor $u$ in $T$, the algorithm deletes $u$ from $T$ and inserts it into $U$. If $v$ does not have a
neighbor in $T$, then $v$ is popped out of $U$ and is moved to $S$. If $U$ is empty, the algorithm chooses
the first vertex of $T$ according to $\sigma$, deletes it from $T$ and pushes it into $U$. The algorithm outputs a spanning forest $F$ of $G$, where the connected components of $G$ and of $F$ coincide, and each component of $F$ is a tree rooted at the first (according to  $\sigma$) vertex of the corresponding component.

We will need the following elementary, but perhaps less known, property of the DFS algorithm. It has been used in particular in \cite{AKS00}, see also \cite{KLS15, R14} for similar arguments.
\begin{prop}\label{dfs}
Let $G=(V,E)$ be a graph, and let $F$ be a forest produced by the DFS algorithm applied to $G$ with some ordering $\sigma$ of the vertices of $G$. If $(u,v)\in E(G)\setminus E(F)$, then one of $\{u,v\}$ is a predecessor of the other in $F$.
\end{prop}

\begin{proof} Since $u,v$ are connected by an edge, they obviously belong to the same connected component $C$ of $G$ (which is also a connected component of $F$, vertex-wise). Assume that $v$ was discovered by the algorithm (and moved to $U$) earlier than $u$. If $u$ is not under $v$ in the corresponding rooted tree $T$ of $F$, then the algorithm has completed exploring $v$ and moved it over to $S$ before getting to explore $u$ (and to move it to $U$) -- obviously a contradiction.
\end{proof}

\medskip
\noindent{\bf Client's win in Client-Waiter games.} We will use the following (reduced form of a) criterion for Client's win in biased Client-Waiter games.
\begin{thmtool}\cite{DK16}\label{CW}
Let $b$ be a positive integer, let ${\cal F}$ be a family of subgraphs of $K_n$. Assume that
$$
\sum_{F\in {\cal F}} \left(\frac{1}{b+1}\right)^{|E(F)|}<\frac{1}{2}\,.
$$
Then in the $(1:b)$ biased Client-Waiter game, played on the edges of $K_n$, Client has a strategy to claim a subset $E_0\subset E(K_n)$ of size $|E_0|\ge \lfloor \binom{n}{2}/(b + 1)\rfloor$ not containing fully the edge set of any $F\in {\cal F}$.
\end{thmtool}

\medskip
\noindent{\bf Dense subgraphs.} This claim shows that every graph contains a relatively dense induced subgraph of prescribed order.
\begin{prop}\label{dense_sub}
Let $k_1,k_2$ be positive integers, and let $G=(V,E)$ be a graph on $|V|=k_1+k_2$ vertices with $m$ edges. Then there exists a subset $R\subset V$ of cardinality $|R|=k_1$, spanning more than $m\left(\frac{k_1-1}{k_1+k_2-1}\right)^2$ edges of $G$.
\end{prop}
\begin{proof} This is a simple averaging argument. Choose a subset $R$ of $k_1$ vertices uniformly at random. Then for a given edge $e\in E(G)$, the probability that both endpoints of $e$ fall into $R$ is
$$
\frac{\binom{k_1+k_2-2}{k_1-2}}{\binom{k_1+k_2}{k_1}}=\frac{k_1(k_1-1)}{(k_1+k_2)(k_1+k_2-1)}>
\left(\frac{k_1-1}{k_1+k_2-1}\right)^2\,.
$$
Then by the linearity of expectation, the expected number of edges of $G$ falling into $R$ is more than $m\left(\frac{k_1-1}{k_1+k_2-1}\right)^2$. Hence there exists a choice of $R$ spanning at least that many edges.
\end{proof}

\section{Proofs of Theorems \ref{th1} and \ref{th2}}\label{sect3}
\noindent{\bf Proof of Theorem \ref{th1}.} Observe first that $G$ has a connected component $C$ on more than $k$ vertices. Indeed, by the theorem's assumption $G$ has no connected components whose sizes fall between $k/2$ and $k$. If every connected component of $G$ has fewer than $k/2$ vertices, then, recalling that $|V(G)|>k$, we can take $W$ to be a maximal by inclusion collection of connected components whose total size is at most $k$. Then $k/2\le|W|\le k$, and $N_G(W)=\emptyset$ -- a contradiction.

Let $T$ be the tree rooted at $r$, obtained by applying the DFS algorithm to $C$, under an arbitrary order of its vertices. Then $|V(T)|=|V(C)|>k$. Apply Proposition \ref{subtrees} to $T$ to find a vertex $v$ and a subset $X$ of its children such that $k/2\le\sum_{u\in X}s(T,u)\le k$. Let $W$ be the union of the subtrees of $T$ rooted at the vertices of $X$. Denote by $P$ the path in $T$ from the root $r$ to $v$. Then $k/2\le |W|\le k$, and by Proposition \ref{dfs} we have: $N_G(W)\subseteq V(P)$. Let $v^*\in V(P)$ be the farthest from $v$ vertex in $N_G(W)\cap V(P)=N_G(W)$. Clearly its distance from $v$ it as at least $|N_G(W)|-1\ge t-1$, by the theorem's assumption. Let $w$ be a neighbor of $v^*$ in $W$. Then the cycle, formed by the path in $T$ from $w$ to $v^*$ (which includes the part of $P$ from $v$ to $v^*$) and the edge $(w,v^*)$, has length at least $t+1$.\hfill $\Box$

\medskip

 \noindent{\bf Remark.} The proof presented is obtained partially in cooperation with Noga Alon.

\bigskip

\noindent{\bf Proof of Theorem \ref{th2}.} Let $G_0$ be a minimal by inclusion non-empty induced subgraph of $G$, satisfying $\frac{|E(G_0)|}{|V(G_0)|}\ge c_1$. Then $G_0$ has the following properties:
\begin{enumerate}
\item $|V(G_0)|>k$;
\item $G_0$ is connected;
\item every subset $W\subseteq V(G_0)$ is incident to at least $c_1|W|$ edges in $G_0$
(otherwise deleting $W$ from $G_0$ produces a smaller subgraph of admissible density).
\end{enumerate}
Let $T$ be the tree rooted at $r$, obtained by applying the DFS algorithm to $G_0$, under an arbitrary order of its vertices. Then $|V(T)|=|V(G_0)|>k$. Apply Proposition \ref{subtrees} to $T$ to find a vertex $v$ and a subset $X$ of its children such that $k/2\le\sum_{u\in X}s(T,u)\le k$. Let $P$ be the path in $T$ from $r$ to $v$. Let $W$ be the union of the subtrees of $T$ rooted at the vertices of $X$. We have: $k/2\le |W|\le k$. Also, by Proposition \ref{dfs}, $N_G(W)\subseteq V(P)$. Denote $U:=N_{G_0}(W), k_1:=|W|, k_2:=|U|$. The set $U\cup W$ spans all edges of $G_0$ incident to $W$, whose number is at least $c_1k_1$, by Property 3) above. Apply Proposition \ref{dense_sub} to $G_0[U\cup W]$ to find a set $R$ of cardinality $|R|=k_1$, spanning more than $c_1k_1\left(\frac{k_1-1}{k_1+k_2-1}\right)^2$ edges. Notice that $|R|=|W|\le k$, and hence by the theorem's assumption $R$ spans at most $c_2|R|=c_2k_1$ edges. Comparing, we obtain:
$$
c_2k_1\ge c_1k_1\left(\frac{k_1-1}{k_1+k_2-1}\right)^2\,,
$$
which after simple arithmetic manipulations produces:
$$
k_2\ge \left(\left(\frac{c_1}{c_2}\right)^{1/2}-1\right)(k_1-1)\ge
 \left(\left(\frac{c_1}{c_2}\right)^{1/2}-1\right)\left(\frac{k}{2}-1\right)\,.
 $$
 Arguing as in the proof of Theorem 1, we obtain a cycle of length at least
 $$
 |N_{G_0}(W)|+1\ge k_2+1\ge \left(\left(\frac{c_1}{c_2}\right)^{1/2}-1\right)\left(\frac{k}{2}-1\right)+1\,,
 $$
 as required.
 \hfill $\Box$

 \bigskip

 \noindent{\bf Remark.} In fact, the proof above utilizes a weaker assumption about local density -- we used only that every subset $R$, whose size is between $k/2$ and $k$, spans at most $c_2|R|$ edges in $G$.

\section{Results about random graphs}\label{sect4}
\noindent{\bf Proof of Proposition \ref{pr1}.} The probability in $G(n,c_1/n)$ that there exists a vertex subset violating the required property is at most
\begin{eqnarray*}
&&\sum_{i\le \delta n}\binom{n}{i}\binom{\binom{i}{2}}{c_2i}\cdot p^{c_2i}\le
 \sum_{i\le \delta n}\left(\frac{en}{i}\right)^i\cdot \left(\frac{eip}{2c_2}\right)^{c_2i}
=\sum_{i\le\delta n}\left[\frac{en}{i}\cdot\left(\frac{ec_1i}{2c_2n}\right)^{c_2}\right]^i\\
&=&\sum_{i\le\delta n}\left[\frac{e^{c_2+1}c_1^{c_2}}{(2c_2)^{c_2}}\cdot\left(\frac{i}{n}\right)^{c_2-1}\right]^i\,.
\end{eqnarray*}

Denote the $i$-th summand of the last sum by $a_i$. Then, if $i\le n^{1/2}$ we get: $a_i\le \left(O(1)n^{-\frac{c_2-1}{2}}\right)^i$, implying $\sum_{i\le n^{1/2}}a_i=o(1)$. For $n^{1/2}\le i\le \delta n$, we have, recalling the expression for $\delta$:
$$
a_i\le \left[\frac{e^{c_2+1}c_1^{c_2}}{(2c_2)^{c_2}}\cdot\left(\frac{c_2}{5c_1}\right)^{c_2}\right]^i\le \left(e\cdot\left(\frac{e}{10}\right)^{c_2}\right)^i=o(n^{-1})\,.
$$
It follows that $\sum_{i\le\delta n}a_i=o(1)$, and the desired property of the random graphs holds {\bf whp}. \hfill$\Box$

\bigskip

\noindent{\bf Proof of Theorem \ref{th4}.} Fix $r\ge 2,\epsilon>0$. Due to monotonicity it is enough to prove the statement for $p=\frac{\psi_r+\epsilon}{n}$. According to the results from \cite{CSW07,FR07}, there exists an $\epsilon'>0$ such that a random graph $G\sim G(n,p)$ contains \whp a subgraph $G_0$ for which $\frac{|E(G_0)|}{|V(G_0)|}\ge r(1+\epsilon')$. We can assume $\epsilon'<\epsilon$. By Proposition \ref{pr1} with $c_1=\psi_r+\epsilon$, $c_2=1+\epsilon'/2$, there is $\delta>0$ such that \whp every $k\le \delta n$ vertices of $G$ span at most $k(1+\epsilon'/2)$ edges. So assume that $G$ has the above stated typical properties. Let $\phi: E(G_0)\rightarrow [r]$ be an $r$-coloring of $E(G_0)$. Choose $E_0\subseteq E(G_0)$ to be a majority color in $\phi$. Then $|E_0|\ge \frac{|E(G_0)|}{r}\ge (1+\epsilon')|V(G_0)|$. Applying Theorem \ref{th2} to the spanning subgraph of $G_0$ with edge set $E_0$ in place of $G$, $c_1=1+\epsilon'$, $c_2=1+\epsilon'/2$, $k=\delta n$, we get a cycle of length at least
$$
\left(\frac{\delta n}{2}-1\right)\left(\left(\frac{1+\epsilon'}{1+{\epsilon'}/{2}}\right)^{1/2} -1\right)+1=\Theta(n)\,,
$$
as required. \hfill$\Box$

\section{Multi-color size Ramsey numbers of paths}\label{sect5}
\noindent{\bf Lower bound.} We prove the following result:
\begin{thm}\label{th7}
 For $r\ge 3$ such that $r-2$ is a prime power, and all sufficiently large $n$, the $r$-color size Ramsey number $\hat{R}(P_n,r)$ satisfies: $\hat{R}(P_n,r)> (r-2)^2n-C\sqrt{n}$, where $C=C(r)$ is a constant depending only on $r$.
\end{thm}

\begin{proof} The argument here has some similarities to the argument about coloring the edges of the complete graph without creating large monochromatic connected components, see \cite{Gya77,Fur81,BirGya87}.
Let $G=(V,E)$ be a graph with $|E|\le (r-2)^2n-C\sqrt{n}$ edges; we can assume $G$ has no isolated vertices. Our aim is to show that there exists an $r$-coloring $\phi$ of $E$ without a monochromatic path on $n$ vertices.

Denote $q=r-2$, and let $A_q$ be an affine plane of order $q$. (The reader unfamiliar with the notion and terminology of affine planes is advised to consult any standard book about finite geometries, say \cite{Dem68} -- or to settle for the following brief description adapted for our purposes. An affine plane of order $q$ is a $q$-uniform hypergraph on $q^2$ vertices (points) with $q^2+q$ edges (lines), such that every pair of points is contained in a unique line, and the lines can be split into $q+1$ disjoint families (ideal points) so that the lines in every ideal point do not intersect. Such a system is known to exist for every prime power $q$.) We assume that the vertex set of $A_q$ is $\{1,\ldots,q^2\}$, and the ideal points are $X_1,\ldots,X_{q+1}$.

Let $V_0=\{v\in V(G): d_G(v)\ge 6r^2\}$. Obviously $|V_0|\le n/3$. Partition $V\setminus V_0$ at random into $q^2$ parts $V_1,\ldots,V_{q^2}$ by putting every vertex $v\in V\setminus V_0$ into $V_x$, $1\le x\le q^2$, independently and with probability $1/q^2$. Given a partition $(V_1,\ldots,V_{q^2})$, we define $\phi$ as follows. If $e$ is an edge of $G$ crossing between $V_x$ and $V_y$ for $1\le x<y\le q^2$, and $L$ is the unique line of $A_q$ passing through $x,y$, then $\phi(e)=i$, where $X_i$ is the ideal point containing $L$. If $e$ falls inside one of the parts $V_x$, color it in an arbitrary color different from $r$ in $\phi$. Finally, color all edges touching $V_0$ by color $r$ in $\phi$.

Observe that the subgraph of the edges colored $r$ under $\phi$ has its covering number bounded by $|V_0|\le n/3$; therefore, this color does not contain the path $P_n$, whose covering number is at least $n/2$. Consider now color $i$, $1\le i\le q+1=r-1$. Every connected subgraph in color $i$ has its vertex set inside $\cup_{x\in L} V_x$, for some line $L$ of $A_q$. (Recall that the lines of the ideal point $X_i$ are pairwise disjoint.) Therefore, it will suffice to argue that with positive probability (under a random partition $V\setminus V_0=V_1\cup\ldots\cup V_{q^2}$), for every line $L$ the corresponding vertex subset $\cup_{x\in L} V_x$ spans fewer than $n-1$ edges.

Consider an edge $e\in E(G)$. For a given line $L$ of $A_q$, the probability that $e\subseteq\cup_{x\in L} V_x$ is exactly $\left(\frac{|L|}{q^2}\right)^2=\frac{1}{q^2}=\frac{1}{(r-2)^2}$. Thus, introducing random variables $A_L=|E(\cup_{x\in L} V_x)|$, we obtain by the linearity of expectation that:
$$
\mathbb{E}[A_L]\le|E(G)|\cdot \frac{1}{(r-2)^2}\le n-\frac{C}{(r-2)^2}\sqrt{n}\,.
$$

Now we need to argue that the random variables $A_L$ are concentrated. Observe that changing the location (part $V_x$) of a vertex $v\in V\setminus V_0$ changes the count in any $A_L$  by at at most $d(v)\le 6r^2$. Hence, applying, say,  the bounded differences inequality of McDiarmid \cite{McD89}, we derive that the probability that the value of $A_L$  deviates from its expectation to reach $n-1$ is at most $\delta(C)$, where $\lim_{C\rightarrow\infty}\delta=0$. Applying the union bound over all lines $L$ of $A_q$ and taking $C=C(r)$ to be large enough, we conclude that with positive probability  a random partition $(V_1,\ldots,V_{q^2})$ satisfies the desired property.
\end{proof}

\noindent{\bf Remark.} In fact essentially the same proof, mutatis mutandis, gives the following Ramsey-type result: let $H$ be a connected graph with $n$ edges, whose covering number $\tau(H)$ satisfies $\tau(H)\gg \sqrt{n}$. Then $\hat{R}(H,r)=\Omega(r^2)n$. The assumption about $\tau(H)$ is essential to argue about concentration of the random variables $A_L$ in the above proof. Some lower bound assumption on the covering number of $H$ is needed here -- if $H$ is the star $K_{1,n}$, then $\hat{R}(H,r)= r(n-1)+1$ (take $G$ to be the star $K_{1,r(n-1)+1}$), providing a linear in $r$ coefficient in front of $n$.

\bigskip

\noindent{\bf Proof of Theorem \ref{th3}.} It will suffice to prove the following result instead:
\begin{thmtool}
For every $C>5$, $r\ge 2$, and all large enough $n$, the $r$-color size Ramsey number $\hat{R}(P_n,r)$ satisfies: $\hat{R}(P_n,r)<400^5Cr^{2+\frac{1}{C-4}}n$.
\end{thmtool}

\begin{proof}
Consider a random graph $G$ drawn from the probability distribution $G(n,p)$ with $p=\frac{Cr}{n}$. With high probability, $G$ has between $\frac{Crn}{3}$ and $Crn$ edges, and satisfies the property stated in Proposition \ref{pr1} with $c_1=Cr$ and $c_2=C/4$. (The value of $\delta$ from that proposition then is $\delta=(20r)^{-\frac{C}{C-4}}$.) So assume that $G$ has indeed these likely properties.

Let $\phi: E(G)\rightarrow [r]$ be an $r$-coloring of $E(G)$. We need to argue that $\phi$ contains a monochromatic long path. Let $E_0$ be a majority color under $\phi$. Obviously, $|E_0|\ge \frac{|E(G)|}{r}\ge \frac{Cn}{3}$. Apply Theorem \ref{th2} with $k=\delta n$, $c_1=\frac{C}{3}$, $c_2=\frac{C}{4}$, to the graph $G_0=([n],E_0)$. We get a cycle (and thus a path) of length at least:
$$
\left(\frac{k}{2}-1\right)\left(\left(\frac{C/3}{C/4}\right)^{1/2}-1\right)+1\ge \frac{k}{20}
= \left(\frac{1}{20r}\right)^{\frac{C}{C-4}}\cdot \frac{n}{20}> \left(\frac{1}{400r}\right)^{\frac{C}{C-4}}n\,.
$$
It is only left to recalculate the parameters. Denote
$$
n_0= \left(\frac{1}{400r}\right)^{\frac{C}{C-4}}n\,.
$$
Then
$$
|E(G)|\le Crn =Cr^{1+\frac{C}{C-4}}400^{\frac{C}{C-4}}n_0\le 400^5Cr^{2+\frac{1}{C-4}}n_0\,,
$$
and $G\rightarrow (P_{n_0})_r$.
\end{proof}

\section{Results about positional games}\label{sect6}
\noindent{\bf Proof of Theorem \ref{th5}.}
We can assume that the constant $\epsilon>0$ is small enough. Maker's strategy is very simple: during the first $(1+\frac{\epsilon}{2})n$ rounds he plays {\em randomly}, i.e. chooses a uniformly random edge to claim out of the set of all available edges at that moment. Observe that in each of these rounds, the number of available edges is at least
$$
\binom{n}{2}-(b+1)\left(1+\frac{\epsilon}{2}\right)n\ge \frac{\epsilon}{3}n^2\,,
$$
meaning that the probability an edge $e$ is chosen by Maker is at most $3/(\epsilon n^2)$, regardless of the history of the game.

We aim to argue that with positive probability Maker creates a relatively dense graph without dense local spots after the first  $(1+\frac{\epsilon}{2})n$ rounds. This is in fact fairly easy to verify: the probability that there exists a set $V_0\subset [n]$ of $|V_0|\le\delta n$ vertices spanning at least $(1+\frac{\epsilon}{4})|V_0|$ edges of Maker (with $\delta=\delta(\epsilon)>0$ being a small enough constant, to be set soon) is at most
$$
\sum_{i\le \delta n}\binom{n}{i}\binom{\binom{i}{2}}{(1+\frac{\epsilon}{4})i}
\left(\left(1+\frac{\epsilon}{2}\right)n\right)^{(1+\frac{\epsilon}{4})i}
\left(\frac{1}{\frac{\epsilon}{3}n^2}\right)^{(1+\frac{\epsilon}{4})i}
$$
(choose a violating set $V_0$ of size $|V_0|=i\le\delta n$ first, then choose a set of $(1+\frac{\epsilon}{4})i$ edges inside $V_0$, then choose in which rounds of the game these edges are claimed by Maker, and finally require that all of these edges are indeed claimed at the corresponding rounds). Continuing the above estimate, we get that the probability in question is at most:
\begin{eqnarray*}
&&\sum_{i\le\delta n}\left[\frac{en}{i}\left(\frac{ei}{2(1+\frac{\epsilon}{4})}\cdot\left(1+\frac{\epsilon}{2}\right)n
\cdot \frac{1}{\frac{\epsilon}{3}n^2}\right)^{1+\frac{\epsilon}{4}}\right]^i\\
&\le&\sum_{i\le\delta n}\left[\frac{en}{i}\left(\frac{5i}{n}\right)^{1+\frac{\epsilon}{4}}\right]^i
\le \sum_{i\le\delta n}\left[20\left(\frac{i}{n}\right)^{\epsilon/4}\right]^i\,.
\end{eqnarray*}
Taking $\delta=25^{-4/\epsilon}$ gives that the above expression, and thus the probability under estimate, are both $o(1)$.

Since the game analyzed is a perfect information game with no chance moves, it is enough to prove that Maker's strategy succeeds to create a locally sparse graph with positive probability. (We proved in fact that his strategy succeeds with
probability approaching 1.)

To summarize, Maker has a strategy to create after the first $(1+\frac{\epsilon}{2})n$ rounds a graph $M$ on $n$ vertices with $|E(M)|=(1+\frac{\epsilon}{2})n$ edges, and such that every set $V_0$ of $i\le \delta n$ vertices spans at most $\left(1+\frac{\epsilon}{4}\right)i$ edges. Invoking Theorem \ref{th2}, we obtain that Maker's graph contains a cycle of length at least
$$
\left(\frac{\delta n}{2}-1\right)\left(\left(\frac{1+\frac{\epsilon}{2}}{1+\frac{\epsilon}{4}}\right)^{1/2}-1\right)+1=\Theta(n)\,,
$$
as required.\hfill$\Box$

\bigskip

\noindent{\bf Proof of Theorem \ref{th6}.}
The proof is quite similar to the previous one, and thus we will be somewhat brief. Due to the bias monotonicity, it is enough to prove the statement for $b=(1-\epsilon)n/2$. Let ${\cal F}$ be the family of all subgraphs of $K_n$ on at most $\delta n$ vertices with density $1+\frac{\epsilon}{2}$. Client's goal will be to avoid claiming any such subgraph in a subset of his edges of size $(1+\epsilon)n$. In order to argue this is feasible, we apply Theorem \ref{CW}. We need to estimate the LHS of the criteria of that theorem. We have:
$$
\sum_{i\le \delta n}\binom{n}{i}\binom{\binom{i}{2}}{(1+\frac{\epsilon}{2})i}\cdot
\left(\frac{1}{(1-\epsilon)\frac{n}{2}+1}\right)^{(1+\frac{\epsilon}{2})i}= o(1)\,,
$$
for small enough $\delta=\delta(\epsilon)>0$.

Hence, it follows from Theorem \ref{CW} that Client has a strategy to claim a set $E_0$ of $|E_0|\ge\binom{n}{2}/(b+1)\ge (1+\epsilon)n$ edges, where every $i\le \delta n$ vertices span at most $(1+\frac{\epsilon}{2})i$ edges of $E_0$. Apply Theorem \ref{th2} to the graph $G_0=([n],E_0)$ to get a linearly long cycle in Client's graph, as desired. \hfill $\Box$

\bigskip

\noindent{\bf Acknowledgement.} The author is grateful to Noga Alon for his contribution to the proof of Theorem \ref{th1}. He would also like to thank Wojciech Samotij for his helpful remarks on the first version of the paper.

\end{document}